\documentclass{amsart}

\usepackage{amsmath}
\usepackage{amsfonts}
\usepackage{amssymb}
\usepackage{comment}
\usepackage{epsfig}
\usepackage{psfrag}
\usepackage{mathrsfs}
\usepackage{amscd}
\usepackage[all]{xy}
\usepackage{rotating}
\usepackage{lscape}
\usepackage{amsbsy}
\usepackage{verbatim}
\usepackage{moreverb}

\usepackage{color}

\newtheorem{theorem}{Theorem}[section]
\newtheorem{prop}[theorem]{Proposition}

\newtheorem{cor}[theorem]{Corollary}

\newtheorem{definition}[theorem]{Definition}

\theoremstyle{remark}
\newtheorem{remark}[theorem]{Remark}
\newtheorem{example}[theorem]{Example}
\newtheorem{q}[theorem]{Question}

\DeclareMathOperator{\Spec}{Spec}

\def\1bord{1\mathrm{Bord}}
\def\2bord{2\mathrm{Bord}}
\def\3bord{3\mathrm{Bord}}

\newcommand{\inv}{{}^{-1}}

\def\cA{\mathcal A}\def\cB{\mathcal B}\def\cC{\mathcal C}
\def\cF{\mathcal F}\def\cH{\mathcal H}
\def\cL{\mathcal L}
\def\cM{\mathcal M}\def\cN{\mathcal N}

\def\cW{\mathcal W}
\def\cZ{\mathcal Z}

\def\CC{\mathbb C}
\def\GG{\mathbb G}
\def\LL{\mathbb L}

\def\ot{\otimes}

\def\ot{\otimes}

\def\ff{\mathfrak f}

\def\wt{\widetilde}

\newcommand{\on}{\operatorname}
\newcommand{\Vect}{\on{Vect}}

\newcommand{\Ind}{\on{Ind}}
\newcommand{\Perf}{\on{Perf}}

\newcommand{\wh}{\widehat}

\newcommand{\Id}{{\rm id}}

\def\RGamma{\on{R}\Gamma}

\def\oo{\infty}

\def\Tr{\mathcal Tr}

\newcommand{\bs}{\backslash}

\def\Coh{\on{Coh}}

\def\Sym{\on{Sym}}

\def\End{\on{End}}

\def\fg{\mathfrak g}

\def\Tr{\on{Tr}}

\def\St{St}

\def\Ind{\on{Ind}}
\def\supp{\on{singsupp}}
\def\Con{\on{Con}}

\newcommand{\Loc}{{\mathcal Loc}}

\newcommand{\Conn}{{\mathcal Conn}}

\newcommand{\SuppSp}{T^{*-1}}
\newcommand{\DCoh}{{\on{DCoh}}}

\newcommand{\QC}{\on{QC}}
\newcommand{\cyl}{Cyl}

\newcommand{\op}{\mathit{op}}
\newcommand{\proper}{\mathit{prop}}

\begin{document}

\title{Betti spectral gluing}
\author{David Ben-Zvi and David Nadler}

\address{Department of Mathematics\\University of Texas\\Austin, TX 78712-0257}
\email{benzvi@math.utexas.edu}
 \address{Department of Mathematics\\University
  of California\\Berkeley, CA 94720-3840}
\email{nadler@math.berkeley.edu}

\maketitle

\newcommand{\GoG}{\displaystyle{\frac{G}{G}}}
\newcommand{\BoB}{\displaystyle{\frac{B}{B}}}
\newcommand{\BGB}{B\bs G/B}
\newcommand{\Gv}{{G^{\vee}}}
\newcommand{\Hv}{{H^{\vee}}}
\newcommand{\Bv}{{B^{\vee}}}

\newcommand{\Tv}{{T^{\vee}}}

\newcommand{\Badj}{{\mathbf{B}}}

\newcommand{\Gadj}{{\mathbf{G}}}

\newcommand{\findim}{\mathit{fd}}

\newcommand\algebra{\mathit{alg}}
\newcommand\geometry{\mathit{geom}}




\begin{abstract}

Given a complex reductive group $G$, Borel subgroup $B\subset G$, and topological surface $S$ with boundary $\partial S$, we study the ``Betti spectral category" 
$\DCoh_\cN(\Loc_G(S, \partial S))$ of coherent sheaves with nilpotent singular support on the character stack of $G$-local systems on $S$ with $B$-reductions along $\partial S$.
Modifications along the components of $\partial S$ endow $\DCoh_\cN(\Loc_G(S, \partial S))$ with commuting actions of the affine Hecke category $\cH_G$ in its realization as coherent sheaves on the Steinberg stack. 
We prove a ``spectral Verlinde formula" identifying  the result of gluing two boundary components  with 
the Hochschild homology of the corresponding $\cH_G$-bimodule structure.
The equivalence is compatible with Wilson line operators  (the action of $\Perf(\Loc_G(S))$ realized by 
Hecke modifications at points)
as well as Verlinde loop operators (the action of the center of $\cH_G$ realized by Hecke modifications along closed loops).
The result reduces the calculation of such ``Betti spectral categories" to the case of disks, cylinders, pairs of pants, and the 
M\"obius band. We also show how to impose arbitrary ramification conditions in terms of modules for the affine Hecke category.

%
%
\end{abstract}



\section{Introduction}
In this paper, we describe structures predicted by four-dimensional topological field theory on the spectral side of the Geometric Langlands correspondence. 
We first state the main result, and then provide some context.

Let $S$ be a 
 (not necessarily oriented) closed
topological surface, $G$ a complex reductive group, 
and $\Loc_G(S)$ the character stack -- the moduli space of $G$-local systems on $S$, considered as a derived stack (i.e., the derived mapping stack $Map(S,BG)$). 
%
Inside of the $-1$st cohomology of its cotangent complex, 
 we single out the nilpotent cone 
 $$
 \xymatrix{
 \cN\subset T^{*,-1}\Loc_G(S)
 }
 $$ consisting of nilpotent endomorphisms. 
 
 Our focus in this paper  is the dg category $\DCoh_\cN(\Loc_G(S))$ of coherent sheaves with nilpotent singular support 
 which sits between perfect complexes and all coherent sheaves 
 $$\xymatrix{
 \Perf(\Loc_G(S))\subset \DCoh_\cN(\Loc_G(S))\subset \DCoh(\Loc_G(S))
 }$$  
 It plays the role of the spectral side in the Betti form of the geometric Langlands conjecture introduced in~\cite{Betti}, in analogy with the de Rham version proposed by Arinkin and Gaitsgory~\cite{AG}.

Now let $S$ be a 
 (not necessarily oriented)
topological surface with boundary $\partial S$, 
$B\subset G$ a Borel subgroup, and 
$\Loc_G(S,\partial S)$ the parabolic derived character  stack
of $G$-local systems on $S$ with $B$-reductions along $\partial S$.

For example, 
in the case of a cylinder $\cyl = S^1 \times [0,1]$ with boundary $\partial \cyl = S^1 \times \{0,1\}$,
we obtain the Grothendieck-Steinberg stack
$$
\xymatrix{
\St_G =   B/B \times_{G/G} B/B \simeq \Loc_G(\cyl, \partial \cyl) 
}
$$ 
and the affine Hecke category in its spectral realization
$$
\xymatrix{
\cH_G=\DCoh(St_G)
}
$$
Here we suppress the nilpotent singular support from the notation since all degree $-1$ codirections turn out to be nilpotent.

The concatenation of cylinders equips $\cH_G$ with a natural monoidal structure,
and by~\cite[Theorem 1.4.6(1)]{spectral} we have a monoidal equivalence
$$
\xymatrix{
\cH_G\simeq \End_{\Perf(G/G)}(\Perf(B/B))
}$$

The preceding example plays a distinguished role as follow. For any surface $S$ with boundary $\partial S$, once we identify a  component of $\partial S$ with the circle, modifications 
of the parabolic structure along that component provide a natural  $\cH_G$-action on $\DCoh_\cN(\Loc_G(S,\partial S))$.
Similarly, once we identify two distinct  boundary components of $\partial S$ with the circle, we obtain
 on the one hand
a natural  $\cH_G$-bimodule structure on $\DCoh_\cN(\Loc_G(S,\partial S))$, and on the other hand,
a new surface $\tilde S =  S/\sim$ where we glue the two boundary components together.

Our main result allows us to recover the dg  category $\DCoh_\cN(\Loc_G(\tilde S,\partial \tilde S))$ as the Hochschild homology category of
the  $\cH_G$-bimodule structure on $\DCoh_\cN(\Loc_G(S,\partial S))$. It is compatible with natural symmetries, realized by Hecke modifications at points and along closed loops,
which we do not state explicitly for now (see Section~\ref{sect loop ops} below).

\begin{theorem}[Corollary~\ref{main cor} below]\label{intro gluing}
There is a canonical equivalence
$$
\xymatrix{
\DCoh_\cN(\Loc_G(\tilde S, \partial \tilde S))\simeq \cH_G \ot_{\cH_G\ot \cH_G^{\op}} \DCoh_\cN(\Loc_G(S,\partial S)))
}$$
respecting commuting Hecke symmetries at points and along closed loops.
\end{theorem}

%
\begin{remark}[Motivation: Betti Geometric Langlands]
This result allows us to reduce the description of the categories attached to arbitrary surfaces to those of elementary building blocks. As explained in~\cite{Betti}, this result suggests a conjectural gluing formula on the automorphic side of the Betti program, providing a ``cut and paste" description of categories of sheaves on moduli stacks of $G^\vee$-bundles on algebraic curves. Such a gluing formula would then reduce the Betti Geometric Langlands conjecture to the case of the thrice-punctured sphere. See Section~\ref{TFT} below for more discussion.
\end{remark}

\begin{remark}
There is a straightforward generalization where $G$ is not necessarily a constant group over $S$ but twisted by automorphisms as one goes around loops of $G$. This arises naturally when $S$ is not orientable and the descent of the constant group $G$ from the two-fold orientation cover is given by an involution on $G$.  
\end{remark}



The theorem is a corollary of the following general assertion.

 Let $p:X\to Y$ and $q:Z\to Y\times Y$ be quasi-smooth morphisms of smooth derived stacks, and set 
$
Z_X = Z \times_{Y \times Y} X \times X.
$
 Assume $p$ is proper.

Introduce the fundamental correspondence
$$
\xymatrix{
Z_X = Z \times_{Y\times Y} X \times X & \ar[l]_-\delta
Z \times_{Y\times Y} X \ar[r]^-p &   Z \times_{Y\times Y} Y
}
$$ 
and the support condition 
$$
\xymatrix{
\Lambda_{-1}=p_*\delta^!\SuppSp_{Z_X} \subset \SuppSp_{Z \times_{Y\times Y} Y}
}
$$

Consider the monoidal category  $\cH =\DCoh(X\times_Y X)$ and the $\cH$-bimodule 
$\DCoh(Z_{X \times X})$.

\begin{theorem}[Theorem~\ref{bimodule traces} below] \label{intro abstract thm}
There is a canonical equivalence of $\Perf(Y)$-modules 
$$\xymatrix{
 \DCoh_{\Lambda_{-1}} (Z \times_{Y \times Y} Y) \simeq \DCoh(Z_X) \otimes_{\cH \otimes \cH^{\op}} \cH
}$$ 
\end{theorem}

The proof of Theorem~\ref{intro abstract thm} is an application of  descent with singular support conditions,
which was developed in our work~\cite{spectral} with Toly Preygel.
The assertion of Theorem~\ref{intro gluing} generalizes the calculation of the Hochschild homology category of $\cH$ itself, 
arising when $S$ is the cylinder, which was the main application of~\cite{spectral}.

\subsubsection{Marked surfaces}\label{spectral ramification}
We note that another instance of Theorem~\ref{intro abstract thm} allows us to specify arbitrary ramification conditions for local systems in terms of modules for $\cH_G$.
Given a quasismooth morphism of smooth stacks $Z\to G/G$, we have a corresponding  stack of ramified local systems
$$\Loc_G(S,\partial S,Z)=\Loc_G(S)\times_{\Loc_G(S^1)} Z$$ which carries a natural singular support condition denoted by $\cN$. 

Now form $\tilde Z = Z  \times_{G/G} B/B$.

\begin{cor}\label{gluing ramification}
There is an equivalence of ramified spectral categories 
$$\Coh_\cN(\Loc_G(S,\partial S,Z))\simeq \Coh(\Loc_G(S,\partial S))\ot_{\cH_G} 
\Coh(\tilde Z) .$$
\end{cor}

(Note that the right hand side is the Hochschild homology of the $\cH_G$-bimodule $\Coh(\Loc_G(S,\partial S))\ot\Coh(\tilde Z)$.)

Thus we can prescribe ramification conditions algebraically or geometrically. For example we recover wild character varieties (Betti spaces of irregular connections) by taking for $Z$
a moduli of Stokes data for irregular connections on the disk with the map to $G/G$ given by taking monodromy.


\subsection{Topological field theory interpretation}\label{TFT}
We include here an informal discussion placing our results within topological field theory (TFT),
and specifically the Geometric Langlands program.
TFTs organize invariants of manifolds that satisfy strong locality properties, 
reducing their calculation to atomic building blocks. We will explain how Theorem~\ref{intro gluing} fits into this paradigm.

Let us first focus on two-dimensional TFTs.
Cutting surfaces along closed curves reduces the calculation of  their TFT invariants to those assigned to the disk, cylinder, and pair of pants (along with the M\"obius band in the unoriented case). 
This information is encoded in a commutative Frobenius algebra structure on the vector space assigned to the circle. For example, from  class functions $\CC[\Gamma/\Gamma]$ on a finite group $\Gamma$,
two-dimensional Yang-Mills theory recovers the orbifold count $\#|\Loc_\Gamma(S)|$ of $\Gamma$-local systems on any surface $S$.

Next let us turn to three-dimensional TFTs, but focus on their two-dimensional invariants. 
Here cutting surfaces along closed curves reduces the calculation of their TFT  invariants to the
 balanced braided tensor structure  on the category assigned to the circle. 
 For example,  from the category  $\Vect[\Gamma/\Gamma]$ of adjoint-equivariant vector bundles on a finite group $\Gamma$, 
Dijkgraaf-Witten theory recovers
 the vector space $\CC[\Loc_G(S)]$  of functions on $\Gamma$-local systems on any surface $S$.

To describe the gluing in more detail,
let $\cZ$ be a three-dimensional TFT, and suppose  the
balanced braided tensor category $\cZ(S^1)$ is presented as a category of modules for an algebra $A$.
Let $S$ be a surface with two boundary components $\partial S_1, \partial S_2$ each identified with 
 $S^1$.
Let  $\wt S$ be the closed surface obtained by gluing together the two   
boundary components $\partial S_1, \partial S_2$ as identified with  $S^1$.
Let $\gamma\subset \wt S$ be the distinguished closed curve given by the glued  boundary components. 
In particular, if $S$  is the disjoint union of two components $S_1, S_2$ each with a single boundary component $\partial S_1 \subset S_1, \partial S_2 \subset S_2$, then  $\wt S\simeq S_1\coprod_\gamma S_2$
and  $\gamma \subset \wt S$ is a separating curve. 
In this case, the invariants $\cZ(S_1)$ and $\cZ(S_2)$ define right and left $A$-modules, and the gluing is given by the tensor product
$$
\xymatrix{
\cZ(\wt S)\simeq \cZ(S_1)\ot_A \cZ(S_2)
}
$$ 
In general, the invariant $\cZ(S)$ is an $A$-bimodule, and the gluing 
is given by the Hochschild homology
$$
\xymatrix{
\cZ(\wt S)\simeq \cZ(S) \ot_{A\ot A^{op}} A
}
$$

Iterating this, for any closed surface $S$ (whose role was played by $\tilde S$ above), one arrives at a complete description of the vector space $\cZ(S)$ assigned  in terms of the 
balanced braided tensor category $\cZ(S^1)$. In particular, compactifying to three-manifolds, the Verlinde formula expresses the dimension $\dim \cZ(S) = \cZ(S \times S^1)$  in terms of the structure constants of the Verlinde algebra $\cZ(S^1\times S^1)$
viewed as the center of the algebra $A$.

Returning to the surface $S$ itself, 
with the choice of a simple closed curve $\gamma\subset S$,
one finds a compatible action of the Verlinde algebra $\cZ(S^1\times S^1)$ on the vector space $\cZ(S)$  as 
``loop operators" -- operators coming from modifications of local systems along $\gamma\subset S$. 
For example,  in the Dijkgraaf-Witten theory of a finite group $\Gamma$,  the action on
 the vector space $\CC[\Loc_G(S)]$  of functions on $\Gamma$-local systems on any surface $S$  results from 
 modifications of  local systems along $\gamma$ as realized by the correspondence
$$
\xymatrix{
\Loc_\Gamma(S) \times \Loc_\Gamma(S^1 \times S^1 )&\ar[l] \Loc_\Gamma(S \coprod_{S\setminus \gamma} S) \ar[r]& \Loc_\Gamma(S)
}$$
where    the torus $S^1\times S^1 $ appears in the unusual, but homotopy equivalent, form of the subspace
of $S \coprod_{S\setminus \gamma} S$ obtained by gluing a tubular neighborhood of $\gamma\subset S$ to itself along the complement of $\gamma$.

%
%
%
%

\subsubsection{Geometric Langlands and four-dimensional TFT}

Kapustin and Witten \cite{KW} discovered that many structures of the Geometric Langlands program fit naturally into the framework of four-dimensional TFT, and more specifically, a topological twist of $\cN=4$ super Yang-Mills. In particular, in its spectral realization,
the invariant assigned to a closed surface $S$ is a category of $B$-branes on the 
moduli $\Loc_G(S)$ of $G$-local systems on $S$. 
To make the link with the Geometric Langlands program more precise,
 one needs to specify the category of $B$-branes. 
 
 In the traditional Geometric Langlands program of Beilinson and Drinfeld \cite{BD}, the role of $S$  is played by a smooth projective complex  algebraic curve $S^{\mathit{alg}}$. One works with the de Rham moduli $\Conn_G(S^{\mathit{alg}})$ of flat $G$-connections on $S^{\mathit{alg}}$. While the analytic stacks underyling $\Conn_G(S^{\mathit{alg}})$ and $\Loc_G(S)$ (for $S$ the surface given by $S^{\mathit{alg}}$ with its classical topology) are equivalent, they have different algebraic structures. 
 On the one hand, categories of quasicoherent sheaves on $\Conn_G(S)$  are not locally constant in the algebraic curve $S$, and so are not the invariants of a TFT.
On the other hand, categories of quasicoherent sheaves on $\Loc_G(S)$   manifestly depend only on the topological surface $S$. 
Moreover, it follows from the results of~\cite{BFN} that the category $\QC(\Loc_G(S))$
of all quasicoherent sheaves,   
 fits into a fully extended $(3+1)$-dimensional oriented TFT.

In the de Rham setup, refining the ideas of Beilinson-Drinfeld, Arinkin and Gaitsgory~\cite{AG} explained that quasicoherent sheaves are too naive 
to be the spectral category in the Geometric Langlands correspondence. Most glaringly, they are not compatible with parabolic induction: the Eisenstein series and constant term constructions fail to give a continuous adjunction.
Arinkin and Gaitsgory developed a beautiful solution to this problem by expanding from quasicoherent sheaves to ind-coherent sheaves with nilpotent singular support.  Moreover, they showed
this category  provides the minimal solution compatible with parabolic induction.
%

Following these developments, to find a spectral category that fits into a TFT, and is rich enough for a topological Geometric Langlands correspondence,
we propose~\cite{Betti} the category $\QC_\cN^!(\Loc_G(S))$ of ind-coherent sheaves with nilpotent singular support on the moduli  of $G$-local systems on $S$, or more concretely,  the small category $\DCoh_\cN(\Loc_G(S))$ formed by its compact objects.
A substantial challenge is that coherent sheaves are much more complicated than perfect complexes:  notably, 
 the functor $\Perf$ takes fiber products to tensor products~\cite{BFN} in reasonable situations, but this typically fails for $\DCoh$. To address this, in the papers~\cite{coherent, spectral}, we developed new techniques to work with coherent sheaves, including descent with prescribed singular support.
The main result of this paper, confirming the  spectral category $\DCoh_\cN(\Loc_G(S))$ enjoys the gluing of a TFT,  is an application of these techniques.

It is an interesting problem to construct a fully extended $(3+1)$-dimensional TFT  
that assigns $\DCoh_\cN(\Loc_G(S))$ to a  surface $S$. 
Results of~\cite{coherent, spectral}, as extended by  the main result of this paper, highlight that such a TFT could assign the 2-category of small $\cH_G$-module categories
to the circle $S^1$. In particular Corollary~\ref{gluing ramification} supports the role of $\cH_G$-modules as ramification conditions for the Betti theory. 
Finding a suitable 3-category to assign to the point is the subject of ongoing work.
\subsection{Summary of sections}

In Sect.~\ref{s:rec}, we collect background material: in Sect.~\ref{s: ss}, we recall Arinkin-Gaitsgory's notion~\cite{AG} of singular support of coherent sheaves; in Sect.~\ref{s: desc}, we recall some  key technical tools: tensor product and descent results from~\cite{spectral};  finally, in Sect.~\ref{Cech}, we explain how to calculate Hochschild homology, specifically in the category of correspondences of derived stacks.  Sect.~\ref{s: main} contains the proof of our main result, Theorem~\ref{bimodule traces}, identifying algebraic and geometric gluing of coherent sheaves on moduli of local systems. Finally, in Sect.~\ref{sect loop ops}, we describe the  natural central  symmetries respected  by  
the equivalence of Theorem~\ref{bimodule traces}.


\subsection{Acknowledgements}

DN would like to thank Zhiwei Yun for many inspiring discussions, including in the context of joint work about  of the Betti version of Geometric Langlands. 
We thank Toly Preygel for his many contributions to our understanding of topics related to the paper. 
We gratefully acknowledge the support of NSF grants DMS-1103525 and DMS-1705110 (DBZ)
and DMS-1502178 (DN). We would also like to thank the anonymous referees for their suggestions. 

\section{Recollections}\label{s:rec}


We work over a fixed ground field $k$ of characteristic zero.  All derived schemes/stacks/etc. are assumed almost of finite type over $k$ (in particular quasi-compact). We will work with both ``small" and ``big" versions of dg categories: the former (including the categories $\DCoh$ and $\Perf$) are objects of the  presentable symmetric monoidal $\infty$-category of small, stable, idempotent complete $k$-linear dg-categories (and exact, $k$-linear functors). The latter (including $\QC^!$ and $\QC$) are objects of the presentable symmetric monoidal $\infty$-category of presentable stable $k$-linear dg-categories -- to distinguish the two we will denote the tensor product of big categories by $\wh{\otimes}$. We refer to~\cite{GR} for a detailed introduction to the relevant notions of derived algebraic geometry.

\subsection{Singular support}\label{s: ss}
We recall here some notions and results from~\cite{AG} (see also~\cite{spectral} for a summary). 

First, recall that a derived scheme $Z$ is quasi-smooth if  it is a derived local complete intersection in the sense
that it is Zariski-locally the derived zero-locus of a finite collection of polynomials.
Equivalently, a derived scheme $Z$ is quasi-smooth if and only if its cotangent complex $\LL_Z$  is perfect of tor-amplitude $[-1,0]$.
More generally, we work with derived stacks that are quasi-smooth in the sense that they admit a smooth atlas of quasi-smooth derived schemes (for example, the character stack is a quotient of a quasi smooth scheme by the action of an affine group). Equivalently, a derived stack admitting a smooth atlas of derived schemes  is quasi-smooth if and only if its  cotangent complex $\LL_Z$  is perfect of tor-amplitude $[-1,1]$.

Let $X$ be a quasi-smooth derived stack and $\LL_X$  its cotangent complex. 
Let $X_{cl}$ denote the underlying  classical stack of $X$. Introduce the shifted cotangent complex
$$
\xymatrix{
\SuppSp_X = \Spec_{X_{cl}} \Sym_{X_{cl}} H^1( \LL_X^\vee)
\simeq (\Spec_{X} \Sym_{X}  \LL_X^\vee[1])_{cl}
}$$
There is a natural affine projection $\SuppSp_X \to X_{cl}$ with fiberwise $\GG_m$-action and the fiber $\SuppSp_X|_x$ at a point $x\in X_{cl}$ is the degree $-1$ cohomology of $\LL_Z|_x$.  We denote by $\{0\}_X\subset \SuppSp_X$ the zero-section.

An important invariant of any $\cF\in\QC^!(X)$ is its singular support 
$$
    \supp \cF \subset \SuppSp_X
    $$ 
It is a conic Zariski-closed subset when $\cF\in \DCoh(X)$, and more generally, for  $\cF\in \QC^!(X)$,   
a union of conic Zariski-closed subsets. For $\cF \in \DCoh(X)$, one has $\supp \cF \subset \{0\}_X$ if and only if $\cF \in \Perf X$. 

Let $\Con X$ denote the set of conic Zariski-closed subsets of $\SuppSp_X$.
    For any $\Lambda \in \Con X$, one defines the full subcategory
        $$ 
        \xymatrix{
        i_{\Lambda} \colon \QC^!_{\Lambda}(X) \ar@{^(->}[r] &  \QC^!(X) 
        }
        $$
         of ind-coherent complexes supported along $\Lambda$. 
    The inclusion $i_{\Lambda}$ admits a right adjoint 
    $$
    \xymatrix{
    \RGamma_{\Lambda}: \QC^!(X) \ar[r] &  \QC_\Lambda^!(X) 
    }
    $$
     We will often regard $\QC^!_{\Lambda}(X)$ as a subcategory of $\QC^!(X)$ via the embedding $i_{\Lambda}$, and 
     likewise regard $\RGamma_\Lambda$ as an endofunctor of  $\QC^!(X)$.

 We set $\DCoh_\Lambda(X) = \DCoh(X) \cap \QC^!_\Lambda(X)$. By \cite[Cor. 8.2.8]{AG}, for 
global complete intersection stacks (in the sense of \cite[Sect. 8.2]{AG}), we have $\QC^!_\Lambda(X) = \Ind \DCoh_\Lambda(X)$. 
  
 We can define functors between categories of sheaves with prescribed singular support by enforcing
 the  support condition:
  
\begin{definition}
Suppose $f\colon X\to Y$ is a map  of quasi-smooth stacks.  

Fix $\Lambda_X\in \Con X$, $\Lambda_Y\in \Con Y$, and define {\em functors with support conditions}
$$
  \xymatrix{
  \ff_* \colon \QC^!_{\Lambda_X}(X) \ar[r] &  \QC^!_{\Lambda_Y}(Y) & \ff^! \colon \QC^!_{\Lambda_Y}(Y) \ar[r] &  \QC^!_{\Lambda_X}(X) 
  }
  $$
  $$ 
\xymatrix{
\ff_* = R\Gamma_{\Lambda_Y} \circ f_* \circ i_{\Lambda_X}  & \ff^! = R\Gamma_{\Lambda_X} \circ f^! \circ i_{\Lambda_Y}
  }
  $$
\end{definition}

\begin{remark}\label{rem: supp unnec}
If the traditional functors preserve support conditions, then the above compositions agree with their
traditional counterparts.
 \end{remark}
 
 Associated to a map $f \colon X \to Y$ is a correspondence
      $$ 
\xymatrix{
\SuppSp_{X} & \ar[l]_-{df^*} \SuppSp_Y \times_Y X \ar[r]^-{\tilde f} & \SuppSp_Y
}
$$ 

Given a subset $U \subset \SuppSp_X$, we may form the subset
      $$
       f_* U = \tilde f ((df^*)^{-1}(U)) \subset \SuppSp_Y
      $$
      If $f:X\to Y$ is proper, then $\tilde f$ is proper, and this defines a map
        $$
        \xymatrix{
         f_* : \Con X \ar[r] &  \Con Y 
         }
         $$
         
 Similarly, given a subset $V \subset \SuppSp_Y$, we may form the subset
      $$
       f^! V =  df^*(X \times_Y V) \subset \SuppSp_X
      $$
      If $f:X\to Y$ is quasi-smooth, then $df^*$ is a  closed immersion, and this defines a map
                            $$
            \xymatrix{
             f^! : \Con Y \ar[r] & \Con X 
             }
             $$

             
\subsubsection{Pushforwards.} 
For $\cF\in \QC^!(X)$, and $f$ schematic and quasi-compact (recall all our stacks are assumed almost of finite type), \cite[Proposition 7.1.3]{AG} ensures
      $$\supp f_*\cF \subset f_*\supp \cF
      $$
      and therefore if 
      $ 
       f_*\Lambda_X \subset \Lambda_Y,
      $ then
      $$
      f_*(\QC^!_{\Lambda_X}(X)) \subset \QC^!_{\Lambda_Y}(Y)
      $$
       and so $\ff_* \simeq f_*$.
      
Following~\cite{spectral}, we codify this condition into a definition:

\begin{definition} Let $X, Y$ be quasi-smooth stacks, and $\Lambda_X\in \Con X, \Lambda_Y\in \Con Y$.

Define a {\em map of pairs  $f\colon (X,\Lambda_X)\to (Y,\Lambda_Y)$} to be a map 
$
f\colon X \to  Y 
$ such that 
$
f_*\Lambda_X \subset \Lambda_Y.
$

In this case, we say ``$f$ takes $\Lambda_X$ to $\Lambda_Y$".
\end{definition}

\begin{remark}


To satisfy the definition of a map of pairs $f\colon (X,\Lambda_X)\to (Y,\Lambda_Y)$, we must have
$$
\xymatrix{ 
(df^*)^{-1}( \Lambda_X) \subset X \times_Y \Lambda_Y
}$$
If $f\colon X\to Y$ is quasi-smooth,  so that $df^*$ is a closed immersion,  this is  equivalent to 
 $$df^*(X \times_Y \SuppSp_Y) \cap \Lambda_X \subset df^* (X \times_Y \Lambda_Y)
 $$
With our previous notation, this  can be rephrased in the form
 $$f^!\SuppSp_Y \cap \Lambda_X \subset f^!\Lambda_Y
 $$
\end{remark}


\subsubsection{Pullbacks.}
      
  Likewise, for $\cF\in \QC^!(Y)$, \cite[Proposition 7.1.3]{AG} ensures
      $$\supp f^!\cF \subset f^! \supp \cF
      $$
        and therefore if 
      $ 
     f^!\Lambda_Y \subset \Lambda_X,
      $ then
      $$
      f^!(\QC^!_{\Lambda_Y}(Y)) \subset \QC^!_{\Lambda_X}(X)
      $$
 
 This condition is implied by the following strong compatibility condition from~\cite{spectral}:

\begin{definition} Let $X, Y$ be quasi-smooth stacks, and $\Lambda_X\in \Con X, \Lambda_Y\in \Con Y$.

Define a {\em strict map of pairs  $f\colon(X,\Lambda_X)\to (Y,\Lambda_Y)$} to be a map 
$
f\colon X \to  Y 
$ such that 
$$
\xymatrix{ 
(df^*)^{-1}( \Lambda_X) = X \times_Y \Lambda_Y
}$$

In this case, we say ``the $f$-preimage of  $\Lambda_Y$ is precisely  $\Lambda_X$".

\end{definition}

\begin{remark}
If $f \colon X\to Y$ is quasi-smooth,  so that $df^*$ is a closed immersion, then 
$f \colon (X,\Lambda_X)\to (Y,\Lambda_Y)$ is a strict map of pairs if and only if
 $$df^*(X \times_Y \SuppSp_Y) \cap \Lambda_X = df^* (X \times_Y \Lambda_Y)
 $$
With our previous notation, this  can be rephrased in the form
 $$f^!\SuppSp_Y \cap \Lambda_X = f^!\Lambda_Y
 $$
\end{remark}


 \subsection{Descent with singular supports}\label{s: desc}
Next, we recall two results from~\cite{spectral}.

The first is the microlocal description of sheaves on fiber products:

\begin{prop}\cite[Proposition 2.1.9]{spectral}\label{fiber product prop}   Let $X_1,X_2$ be 
quasi-smooth stacks over a smooth separated base $Y$. Then the functor
of external product over $Y$ induces an equivalence
$$\xymatrix{\DCoh(X_1) \otimes_{\Perf(Y)} \DCoh(X_2)  \ar[rr]_-{\boxtimes_Y}^-\sim && 
 \DCoh_{\Lambda}(X_1 \times_Y X_2) \subset 
 \DCoh(X_1 \times_Y X_2)\\
 }$$
where $\Lambda = i^!(\SuppSp_{X_1\times X_2})$ for 
$i: X_1 \times_Y X_2 \to X_1 \times X_2$.
\end{prop}

The most significant result of~\cite{spectral} we will need is descent for sheaves with prescribed singular support.

\begin{definition} A {\em strict Cartesian diagram of pairs} is a Cartesian diagram of quasi-smooth stacks which is also a commutative diagram of maps of pairs
\[ \xymatrix{
  (Z = X \times_S X', \Lambda_Z) \ar[r]^-{p_2} \ar[d]_{p_1} &  (X', \Lambda_{X'}) \ar[d]^{q} \\
  (X, \Lambda_X) \ar[r]_{p} & (Y, \Lambda_Y) } \]
satisfying the strictness condition
$$\Lambda_Z \supset p_1^! \Lambda_X \cap p_2^! \Lambda_{X'}$$
\end{definition}

\begin{theorem}\cite[Theorem 2.4.1, Corollary 2.4.2] {spectral}\label{descent theorem} 
Suppose $f \colon (X_\bullet, \Lambda_\bullet) \to (X_{-1}, \Lambda_{-1})$ is an augmented simplicial diagram of maps of  pairs with all stacks quasi-smooth and maps proper.  Suppose further that:
  \begin{enumerate}
      \item The face maps are quasi-smooth.
       \item For any map $g:[m] \to [n]$ in $\Delta_+$, the induced commutative square
 \[
\xymatrix{
(X_{n+1},\Lambda_{n+1}) \ar[d]_-{\tilde g} \ar[r]^-{d_0}& (X_{n},\Lambda_{n}) \ar[d]^-{g}  \\
(X_{m+1}, \Lambda_{m+1})   \ar[r]^-{d_0}& (X_{m},\Lambda_{m}) } \]
is a strict Cartesian diagram of pairs.

      \item
      Pullback along the augmentation
      $$
      \xymatrix{
      \ff^! \colon \QC^!_{\Lambda_{-1}}(X_{-1}) \ar[r] & \QC^!_{\Lambda_0}(X_0)
      }
      $$  is conservative. 
      
      \item Each $\QC^!_{\Lambda_k}(X_k)$ is compactly generated for each $k \geq 0$.  
  \end{enumerate}
  
  Then $\QC^!_{\Lambda_{-1}}(X_{-1})$ is compactly generated as well, and
  pushforward along the augmentation provides an equivalence 
\[ \xymatrix{
   \DCoh_{\Lambda_{-1}}(X_{-1}) & \ar[l]_-{\sim}   \left| \DCoh_{\Lambda_\bullet}(X_\bullet), \ff_{\bullet*} \right| 
    }\]
in the $\infty$-category of small dg categories.
 \end{theorem}

\subsection{Bar  and \v Cech constructions}\label{Cech}
Let us now recall the relative bar construction in algebra and geometry (see~\cite{BFN} for a review in the $\oo$-categorical setting). 

Let $\cC$ be a symmetric  monoidal $\oo$-category. Given an algebra $\cA\in \cC$, the trace
of an $\cA$-bimodule $\cM \in \cC$  is defined to be the tensor product of bimodules
$$
\xymatrix{
\Tr(\cA,\cM) =  \cM \otimes_{\cA \otimes \cA^{\op}} \cA 
}
$$

Suppose $\cB\to\cA$ is a morphism of algebra objects. Viewing $\cA$ as an algebra
in $\cB$-bimodules, we can identify $\cA$ with the geometric realization of the relative bar resolution
  \[ 
  \xymatrix{
  \cA \simeq
  \left| \cA^{\otimes_\cB (\bullet+2)} \right|.
  }
  \]
  Note the two extreme cases: when $\cB = \cA$, then we recover the constant resolution;
when $\cB$ is the monoidal unit, we recover the absolute bar resolution 
$$
\xymatrix{
\cA\simeq \left| \cA^{\otimes (\bullet+2)} \right|.
}$$

 The relative bar resolution can be used to calculate the trace
  \[   
  \xymatrix{
\Tr(\cA, \cM) =   \cA \otimes_{\cA \otimes \cA^{op}} \cM 
\simeq \left| \cA^{\otimes_\cB (\bullet+2)} 
  \right| \otimes_{\cA \otimes \cA^{op}} \cM 
\simeq \left| \cA^{\otimes_\cB (\bullet+2)} 
   \otimes_{\cA \otimes \cA^{op}} \cM \right|
}  
\]

Given a correspondence $Y\leftarrow Z \to Y$ of derived stacks, i.e.~a map $Z\to Y\times Y$ of derived stacks,  its geometric trace is
defined to be the fiber product
$$
\xymatrix{
\Tr^\geometry(Y, Z)= Z\times_{Y\times Y} Y.
}$$

Given a map   $p: X\to Y$   of derived stacks, we can form its \v Cech construction 
$$\xymatrix{
X_\bullet= X^{\times_{Y}(\bullet+1)} \ar[r] &  Y.
}$$ 
viewed as an augmented simplicial object.
In general, this is not a colimit diagram, but we will only encounter situations where it is.

Note that we can identify the \v Cech construction of the base change
$$
\xymatrix{
Z\times_{Y\times Y} X \ar[r] & Z\times_{Y\times Y} Y 
}
$$
with the
substitution of the \v Cech construction of $p:X\to Y$ into the definition of the trace
$$\xymatrix{
Z \times_{Y \times Y} X^{\times_{Y}(\bullet+1)} \ar[r] &  \Tr^\geometry(Y, Z)
}$$ 
Again, in general,  this is not a colimit diagram, but we will only encounter situations where it is.

\subsubsection{Informal discussion: \v Cech as bar.} To guide later discussion, let us informally
relate the bar and \v Cech constructions.
We will work in this section in the category of correspondences of derived stacks, with objects derived stacks and morphisms correspondences
of derived stacks (though at no point will we need to calculate colimits in the correspondence category). 

 Any derived stack $Y$ is naturally an algebra object in the correspondence category with multiplication 
$$
\xymatrix{
Y \times Y & \ar[l]_-\delta Y \ar[r]^-{\Id_Y} & Y
}
$$  
More generally, any map  $q:Z\to Y$ of derived stacks provides a $Y$-module
with action 
 $$
\xymatrix{
Y \times Z & \ar[l]_-{q\times \Id_Z} Z \ar[r]^-{\Id_Z} & Z
}
$$  

Given a map $p:X\to Y$, the fiber product $X\times_Y X$ is also 
  an algebra object with multiplication 
$$
\xymatrix{
X \times_Y X \times X \times_Y X & \ar[l]_-\delta X \times_Y X \times_Y X \ar[r]^-{p_{13}} & X \times_Y X
}
$$  
The relative diagonal $X \to X\times_Y X$ is a map of algebra objects, and
$X\times_Y X$ descends to an algebra object in $X$-bimodules
 with multiplication 
$$
\xymatrix{
X \times_Y X \times_X X \times_Y X & \ar[l]_-\sim X \times_Y X \times_Y X \ar[r]^-{p_{13}} & X \times_Y X
}
$$  
Note that here the multiplication can be viewed as an honest map.

Given a correspondence $Z\to Y\times Y$, note that its algebraic and geometric traces agree
$$
\xymatrix{
\Tr(Y, Z) \simeq Z\times_{Y \times Y} Y = \Tr^\geometry(Y, Z) 
}
$$

Now consider the  $X\times_Y X$-bimodule given by the base change
$$
\xymatrix{
Z_X = Z \times_{Y\times Y} X \times X
}
$$ 
Let us calculate its trace  $\Tr(X\times_Y X, Z_X)$ using
the relative bar resolution 
  \[ 
  \xymatrix{
  X\times_Y X \simeq
  \left| (X\times_Y X)^{\times_X (\bullet+2)} \right|
  \simeq 
   \left| X^{\times_Y (\bullet+3)} \right|
  }
  \]
  of the map of algebras $X\to X\times_Y X$:
we find $$\xymatrix{ Z_X\times_{(X\times_Y X)^2} (X\times_Y X)^{\times_X (\bullet+2)}\simeq Z_X \times_{X\times X}  X^{\times_{Y}(\bullet+1)} \simeq Z \times_{Y \times Y} X^{\times_{Y}(\bullet+1)}.}$$
(Note the geometric realization here is taken in the category of morphisms of stacks, not of correspondences -- we will not need to apply colimits in the correspondence category.)
We identify the result with the \v Cech construction of  the map 
$$
\xymatrix{
Z\times_{Y\times Y} X \ar[r] & Z\times_{Y\times Y} Y 
}
$$
but with the alternative augmentation
$$\xymatrix{
Z \times_{Y \times Y} X^{\times_{Y}(\bullet+1)} \ar[r] & \Tr(X\times_Y X, Z_X)
}$$ 

In situations where the \v Cech construction calculates  $\Tr(Y, Z) \simeq Z \times_{Y \times Y} Y$,
we then have  Morita-invariance  of the trace 
$$
\xymatrix{
\Tr(X\times_Y X, Z_X)\simeq \Tr(Y, Z)
}
$$
We will only encounter sitations where this holds, but will pass to categories of sheaves where an interesting failure
of Morita-invariance occurs in the form of singular support conditions.


\section{Gluing geometric bimodules}\label{s: main}
We now prove our main theorem, a gluing result for geometric bimodules. We will use the notation of Section~\ref{Cech}.

 Let $p:X\to Y$ and $q:Z\to Y\times Y$ be quasi-smooth morphisms of smooth derived stacks, and set 
$
Z_X = Z \times_{Y \times Y} X \times X.
$
 Assume $p$ is proper.
 
 Let $\cZ_{-1}$ denote the geometric trace $Z\times_{Y\times Y} Y$ of the $Y$-bimodule $Z$. 

Recall the fundamental correspondence
$$
\xymatrix{
Z_X = Z \times_{Y\times Y} X \times X & \ar[l]_-\delta
Z \times_{Y\times Y} X \ar[r]^-p &   Z \times_{Y\times Y} Y= \cZ_{-1}
}
$$ 
and introduce on $\cZ_{-1}$ the support condition 
$$
\xymatrix{
\Lambda_{-1}=p_*\delta^!\SuppSp_{Z_X}
}
$$

Introduce the monoidal category  $\cH =\DCoh(X\times_Y X)$ and the $\cH$-bimodule 
$\DCoh(Z_{X})$.

\begin{theorem}\label{bimodule traces}
With the above assumptions and constructions, there is a canonical equivalence of $\Perf(Y)$-modules 
$$\xymatrix{
\Tr(\cH,\DCoh(Z_X))\simeq \DCoh_{\Lambda_{-1}} (\cZ_{-1})
}$$ 
\end{theorem}

\begin{proof}
We would like to compare sheaves on the diagram $$\cZ_\bullet=Z \times_{Y \times Y} X^{\times_{Y}(\bullet+1)}\simeq Z_X \times_{X\times X} (X\times_Y X)^{\times_X \bullet} $$ with, on the one hand, the category $\DCoh_{\Lambda_{-1}}(\cZ_{-1})$ 
and, on the other hand, the trace of the $\cA=\cH$-bimodule $\DCoh(Z_X)$ as calculated via the bar construction relative to $\cB=\Perf(X)$, following the general \v Cech - vs - bar format from the previous section. (We would like to emphasize that we work in the category of morphisms of stacks, not correspondences.) The face maps in the simplicial diagram $\cZ_\bullet$ are all proper and quasi-smooth maps, being base changes of the proper and quasi-smooth map $\pi$. The degeneracy maps (given by relative diagonals) are likewise proper since $\pi$ is representable and separated.
Let $$q_\bullet:\cZ_\bullet\simeq Z_X \times_{X\times X} (X\times_Y X)^{\times_X \bullet}  \to \cW_\bullet=  Z_X \times (X\times_Y X)^{\times \bullet}$$ be the 
map to the absolute two-sided bar construction, and define $$\Lambda_\bullet=q_\bullet^! \SuppSp {\cW_\bullet}$$ to be the resulting support condition on $\cZ_\bullet$, so that we have a simplicial diagram of pairs $(\cZ_\bullet, \Lambda_\bullet)$. 

We now pass to categories using $(\DCoh_\Lambda, f_*)$, obtaining an augmented simplicial category $$\cC_\bullet= 
\DCoh_{\Lambda_{\bullet}}(\cZ_\bullet)\to \DCoh_{\Lambda_{-1}}(\cZ_{-1}).$$

 By repeated application of Proposition~\ref{fiber product prop}, we have the identification
 $$\DCoh_{\Lambda_n}(\cZ_n)\simeq \DCoh(Z_X) \ot_{\Perf(X\times X)} \cH^{\ot_{\Perf(X)}n}$$
on simplices compatibly with structure maps, and thus an identification of simplicial objects 
$$
\xymatrix{
\cC_\bullet =  \DCoh(Z_X)\otimes_{\Perf(X\times X)} \cH^{\otimes_{\Perf(X)} \bullet}\simeq  
\DCoh(Z_X)\otimes_{\cH\otimes\cH} \cH^{\otimes_{\Perf(X)} \bullet+2}
}$$
with the relative bar construction. Thus we have identified $$\left| \cC_\bullet \right|\simeq \Tr(\cH, \DCoh(Z_X)).
$$

We will now verify the hypotheses of Theorem~\ref{descent theorem} are satisfied for the augmented simplicial diagram
$$\xymatrix{
(\cZ_\bullet, \Lambda_\bullet) \longrightarrow (\cZ_{-1},\Lambda_{-1})
}
$$
As already noted, the face maps are quasi-smooth and proper, the degeneracy maps are proper, and the requisite squares are Cartesian.
Next, note that $p$ is a representable
proper map, so that applying \cite[Prop. 7.4.19]{AG}, we see the augmentation is conservative,
since by definition
the support condition on the target $\cZ_{-1}$ is the image of the support condition on the source $\cZ_0$.
Next, we need to see that the categories $\QC^!_{\Lambda_n}(\cZ_n)$ are compactly generated for each $n\geq 0$. 
First, we can identify $\QC^!_{\Lambda_n}(\cZ_n)$ as the essential image of 
$$\QC^!(Z_X)\wh{\otimes}_{\QC(X\times X)} \QC^!(X\times_Y X)^{\wh{\otimes}_{\QC(X)} \bullet}\to \QC^!(\cZ_n)$$ (where $\wh{\otimes}$ denotes the tensor product of stable presentable $k$-linear categories). 
This follows directly from~\cite[Proposition 7.4.12]{AG} (as in the proof of Proposition~\ref{fiber product prop} given in~\cite[Proposition 2.1.9]{spectral}.)
Since $\QC^!(Z_X)$, $\QC^!(X)$, and $\QC^!(X\times_Y X)$ are compactly generated and all structure maps preserve compact objects by our hypotheses, it follows that the left hand side is compactly generated, hence so is $\QC^!_{\Lambda_n}(\cZ_n)$.

It remains to establish that the diagram $(\cZ_\bullet,\Lambda_{\bullet})$ is a {\em strict} diagram of pairs, which 
we now prove separately as Proposition~\ref{strictness}.
\end{proof}

\begin{prop}\label{strictness} The diagram $(\cZ_\bullet,\Lambda_{\bullet})$ is a strict Cartesian diagram of pairs.
\end{prop}

\begin{proof}
The proof closely mimics the proof of \cite[Proposition 3.3.8]{spectral}, which is the case $Z=Y.$ We indicate the idea and modifications necessary for the general case.
 
We give an explicit description of the shifted cotangents to $\cZ_n$, on the level of  geometric points of the derived stack. Such points can be represented by tuples $$(y,\{x_0,\dots, x_n\}, z,\gamma)$$ with $y\in Y$, $x_i\in p\inv y\subset X$, $z\in Z$ with $\mu_l(z)=y$ and $\gamma:\mu_l(z)\sim \mu_r(z)$, and  $\mu_l(z)=\mu_r(z)=y$. Here we denote by $\mu_l\times \mu_r:Z\to Y\times Y$ the defining projection. We represent points of $\cW_n$ by tuples $$(y_0,x_0,x_0'; \dots, y_{n-1},x_{n-1}, x_{n-1}'; z, x_n, x_n')\; : \; p(x_i)=p(x_i')=y_i, \; \mu_l(z)= p(x_n), \; \mu_r(z)=p(x_n').$$ The map $q_n:\cZ_n\to \cW_n$ is thus represented by $$q_n(y,\{x_0,\dots, x_n\}, z)= (y,x_0,x_1; y,x_1,x_2, \dots, y,x_{n-1},x_n; z,x_n,\gamma\circ x_0)$$ where we use the path $\gamma$ to identify $\mu_r(z)\sim p(x_0)$. 

Under these identifications, we write at a geometric point $\eta=(y,\{x_0,\dots, x_n\}, z,\gamma)$ of $\cZ_n$
$$\xymatrix{\SuppSp_{\cZ_n}|_{\eta} \ar[r]^-\sim &  \{v_0,\dots, v_{n+1}\in \Omega_Y \; : \; &
dp^*_{x_1} v_0= dp^*_{x_1} v_1, 
&\dots\\ 
&dp^*_{x_n} v_{n-1}= dp^*_{x_n} v_n,
&d(\mu_l)^*_z v_n= d(\mu_r)^*_z v_{n+1},  
&dp^*_{x_0}d\gamma^* v_{n+1}= dp^*_{x_0} v_0 \} }$$
while at a geometric point $\eta'=(y_0,x_0,x_0'; \dots, y_{n-1},x_{n-1}, x_{n-1}'; z, x_n, x_n')$ of $\cW_n$
$$\xymatrix{\SuppSp_{\cW_n}|_{\eta'} \ar[r]^-\sim 
&  \{v_0,\dots, v_{n+1}\in \Omega_Y \; : \; &
dp^*_{x_0} v_0= 0= dp^*_{x_0'} v_0, 
&\dots\\ 
dp^*_{x_{n-1}} v_{n-1}=0= dp^*_{x_{n-1}'} v_{n-1}, 
&d(\mu_l)^*_z v_n= d(\mu_r)^*_z v_{n+1},  
&dp^*_{x_n} v_{n}= 0= dp^*_{x_n'} v_{n+1} \} }$$

Combining these descriptions, we find at a geometric point $\eta=(y,\{x_0,\dots, x_n\}, z,\gamma)$ of $\cZ_n$
$$\xymatrix{\Lambda_n|_{\eta} \ar[r]^-\sim &  \{v_0,\dots, v_{n+1}\in \Omega_Y \; : \; &
dp^*_{x_1} v_0= 0= dp^*_{x_1} v_1, 
&\dots\\ 
&dp^*_{x_n} v_{n-1}= 0= dp^*_{x_n} v_n,
&d(\mu_l)^*_z v_n= d(\mu_r)^*_z v_{n+1},  
&dp^*_{x_0} d\gamma^*v_{n+1}=0= dp^*_{x_0} v_0 \} }$$

We now need to check for any $\psi \colon [m] \to [n] \in \Delta$ that the induced
diagram
    \[ \xymatrix{
    (\cZ_{n+1},\Lambda_{n+1}) \ar[d]^{\tilde g} \ar[r]^-{d_0}& (\cZ_n,\Lambda_{n}) \ar[d]^g  \\
    (\cZ_{m+1}, \Lambda_{m+1})   \ar[r]^-{d_0}& (\cZ_m,\Lambda_m) } \]
    is a strict Cartesian diagram of pairs, in other words that for any geometric point $\eta$ we have 
    $$((d_0)^! \Lambda_n)|_\eta \cap ({\tilde g}^! \Lambda_{m+1})|_\eta \subset \Lambda_{n+1}|_\eta$$

We first consider the case of face maps, i.e., of $\psi$ an inclusion. The simplicial map 
$\wt{\psi}:[m+1]\to [n+1]$ inducing $\wt{g}$
is given by $\wt{\psi}(0)=0$, $\wt\psi(i)= 1+ \psi(i-1)$ for $i\geq 1$. 
 In this case the support condition
 $({\tilde g}^! \Lambda_{m+1})|_\eta$ consists of the one equation $d(\mu_l)^*_z v_{m+1}= d(\mu_r)^*_z v_{m+2}$ coming from $Z$   and the subset of the equations $dp^*_{x_i} v_{i-1}= 0= dp^*_{x_i} v_i$ corresponding to indices $i$ in the image of $\wt\psi$,  together with additional degeneracy identities among the complementary $v_j$. Likewise the support condition $((d_0)^! \Lambda_n)|_\eta$ consists of the $Z$-equation and the equations $dp^*_{x_i} v_{i-1}= 0= dp^*_{x_i} v_i$ for $i\geq 1$, plus a degeneracy condition relating $v_0$ and $v_{n+1}$. Since $\wt\psi$ has $0$ in its image, the intersection of these two conditions imposes all the equations defining $\Lambda_{n+1}$, as desired.
 
The general case follows the argument of \cite[Proposition 3.3.8]{spectral} verbatim. We factor $\psi:[m]\to [n]$ (in a unique fashion) 
    \[ 
    \xymatrix{
    \psi \colon [m] \ar@{->>}[r]^-{\pi} &  [k]\simeq im(\psi) \ar@{^(->}[r]^-{\iota} &  [n] 
    }
    \]   
    as a 
    surjection followed by an injection
    This gives rise to an extended diagram
    \[ \xymatrix{
    (\cZ_{n+1},\Lambda_{n+1}) \ar[d]^{\tilde p} \ar[r]^-{d_0}& (\cZ_n,\Lambda_{n}) \ar[d]^p  \\
    (\cZ_{k+1},\Lambda_{k+1}) \ar[d]^{\tilde q} \ar[r]^-{d_0}& (\cZ_k,\Lambda_{k}) \ar[d]^q  \\
    (\cZ_{m+1}, \Lambda_{m+1})   \ar[r]^-{d_0}& (\cZ_m,\Lambda_m) } \]
    where $p$ correspond to the injection $\iota$, and $q$ corresponds to the surjection $\pi$.  
    
   We need to show that the large square satisfies the required strictness. By the case of a surjection, 
   we know that the top square satisfies the required strictness.  
    Thus it suffices to show that $(\tilde q)^! \Lambda_{m+1}$ already equals $\Lambda_{k+1}$ since then 
    \[
     (\tilde q \circ \tilde p)^! \Lambda_{m+1} = (\tilde p)^! (\tilde q)^! \Lambda_{m+1} = (\tilde p)^! \Lambda_{k+1} 
    \] 

      Define $\pi'\colon [k] \to [m]$ to be the section of $\pi$ given by its break points
    \[ \pi'(i) = \sup \pi^{-1}(i)  \]  
      Thus the pullback map admits the description
    \[ \xymatrix{
  (v_0, \ldots, v_{m+1}) \ar@{|->}[r] &  (v_0, v_{1+\pi'(0)}, \ldots, v_{1+\pi'(k)})
  } \] and thus itself admits a section by repeating terms. 
        
     It is now elementary to see that $(\tilde q)^! \Lambda_{m+1} = \Lambda_{k+1}$:  the inclusion 
     $(\tilde q)^! \Lambda_{m+1} \subset \Lambda_{k+1}$
     is evident, while the inclusion $(\tilde q)^! \Lambda_{m+1} \supset \Lambda_{k+1}$ follows from the fact that the 
     noted section takes $\Lambda_{k+1}$ into $\Lambda_{m+1}$. This completes the proof.  
\end{proof}


\section{Gluing parabolic local systems}

Let us introduce the notation $\Gadj=G/G \simeq \cL BG \simeq \Loc_G(S^1)$ and $\Badj=B/B \simeq\cL BG   \simeq \Loc_G(S^1)$ for the adjoint quotients, and $p:\Badj\to \Gadj$ for the Grothendieck-Springer resolution.

For a closed (not necessarily orientable) surface with boundary $S$, 
consider the restriction of local systems to the boundary 
$$
\xymatrix{
\Loc_G(S) \longrightarrow \Loc_{G}(\partial S)\simeq (\Gadj)^{\pi_0(\partial S)}
}$$ 
where the  isomorphism depends on an identification of each boundary component with $S^1$.

Write $\partial S = \coprod_{\alpha\in \pi_0(\partial S)} \partial_\alpha S$ for the decomposition of $\partial S$
into connected components. For $A \subset \pi_0(\partial S)$, 
denote by
$\partial_A S = \coprod_{\alpha\in A} \partial_\alpha S$
 the union of those connected components.

Define the stack of parabolic local systems to be the base change
$$
\xymatrix{
\Loc_G(S,\partial_A S)= \Loc_G(S) \times_{\Loc_{G}(\partial_A S)} \Loc_{B}(\partial_A S)
\simeq  \Loc_G(S) \times_{ (\Gadj)^{A}}  (\Badj)^{A}
}$$
so in other words, the stack of local systems with a Borel reduction along $\partial_A S$. 

Define the parabolic spectral category to be
$$\xymatrix{
\DCoh_\cN(\Loc_{G}(S,\partial_A S))
}$$  

\begin{example}
The Steinberg stack  
$$
\xymatrix{
\St_{G}= \Badj \times_{\Gadj}  \Badj \simeq \cL(B\bs G/B) \simeq\Loc_G(\cyl, \partial \cyl)
}$$
is the special case of the cylinder $(\cyl=S^1\times [0,1], \partial \cyl = S^1\times \{0,1\})$.

It carries an $(S^1\times S^1)$-action separately rotating the boundary components, with the diagonal rotation identified with the rotation of the cylinder.

The affine Hecke category is the corresponding parabolic spectral category 
$$
\xymatrix{
\cH_{G}= \DCoh(\St_{G})\simeq \DCoh_\cN(\Loc_G(\cyl,\partial \cyl))
}$$ 
since 
all degree $-1$  codirections of $\St_G$ are nilpotent. \end{example}


For $A\subset \pi_0(\partial S)$, define a marking of $\partial_A S$ to be the data of a marked point $x_\alpha\in \partial_\alpha S$ and orientation of $\partial_\alpha S$, for $\alpha\in A$. Note that an orientation of $S$ can be used to induce an orientation of $\partial S$
all at once.

A marking of $\partial_A S$ provides  identifications  $\partial_\alpha S\simeq S^1$, for  $\alpha\in A$,
up to contractible choices. 
Given two distinct $\alpha \not = \beta\in A$,  
set $\wt A = A \setminus \{\alpha, \beta\}$,
and  introduce the glued surface 
$$
\xymatrix{
\wt{S}=S \coprod_{\partial_\alpha S \coprod \partial_\beta S} S^1 
}
$$
where we identify the two corresponding boundary components.
Note that the 
image of the glued circles provides
 a canonical circle $\gamma:S^1\hookrightarrow S$ in the interior (or, better, a canonical cylinder $S^1\times [0,1]$ in $S$ with meridian $\gamma$).
 
Passing to local systems, we obtain the presentation
$$
\xymatrix{
\Loc_G(\wt S, \partial_{\wt A} \wt S)\simeq \Loc_G(S, \partial_{\wt A} S) \times_{\Gadj \times \Gadj} \Gadj
}
$$

Observe that the 
spectral category
$\DCoh_\cN(\Loc_G(\wt S, \partial_{\wt A} \wt S))$ is naturally a module over 
$$
\xymatrix{
\Perf(\Loc_G(\wt S, \partial_{\wt A} \wt S)) \simeq
\Perf( \Loc_G(S, \partial_{\wt A} S) )\otimes_{\Perf(\Gadj \times \Gadj)} \Perf(\Gadj)
}
$$


Now recall that the standard convolution diagrams equip the  affine Hecke category
$$
\xymatrix{
\cH_{G}= \DCoh(\St_{G})\simeq \DCoh_\cN(\Loc_G(\cyl,\partial \cyl))
}$$ 
with a monoidal structure compatible with rotations of the cylinder.
By \cite[Theorem 1.4.6(1)]{spectral}, we have a monoidal equivalence 
$$
\xymatrix{
\cH_G\simeq \End_{\Perf(\Gadj)}(\Perf(\Badj))
}$$
compatible with rotations of the cylinder 
on the left hand side and rotations of loops on the right hand side. 
Geometrically, the monoidal structure is realized by gluing cylinders along consecutive boundary components. 
We will use the orientation-reversing diffeomorphism  of the cylinder
 given by reversing the interval to fix an equivalence of the  affine Hecke category with its opposite algebra.

%
%
%
For $A\subset \pi_0(\partial S)$,  a marking of $\partial_A S$ equips $\DCoh_\cN(\Loc_{G}(S,\partial_A S))$
with the structure of $\cH_{G}^{\otimes A}$-module.
In particular, an ordered pair of  distinct 
 $\alpha \not = \beta\in A$ equips $\DCoh_\cN(\Loc_{G}(S,\partial_A S))$ with the structure of $\cH_G$-bimodule.

Observe that the resulting trace
$$
\xymatrix{
\Tr(\cH_G, \DCoh_\cN(\Loc_G(S,\partial_A S)))= \DCoh_\cN(\Loc_G(S,\partial_A S)) \ot_{\cH_G\ot\cH_G^{\op}} \cH_G
}
$$ 
is naturally a module over
$$
\xymatrix{
\Perf(\Loc_G(\wt S, \partial_{\wt A} \wt S)) \simeq
\Perf( \Loc_G(S, \partial_{\wt A} S) )\otimes_{\Perf(\Gadj \times \Gadj)} \Perf(\Gadj)
}
$$

\begin{cor}\label{main cor} There is a canonical equivalence of $\Perf(\Loc_G(\wt S, \partial_{\wt A} \wt S))$-modules 
$$
\xymatrix{
\Tr(\cH_G, \DCoh_\cN(\Loc_G(S,\partial_A S))) \simeq \DCoh_\cN(\Loc_G(\wt{S}, \partial_{\wt A} \wt S))
}$$
between the trace of the parabolic spectral category and the spectral category of the glued surface.
\end{cor}

\begin{proof} We will apply Theorem~\ref{bimodule traces} with $X = \Badj = B/B$, $Y = \Gadj=G/G$, and $Z=
\Loc_G(S, \partial_{\wt A} S)$.  It thus suffices to identity the support 
condition $\Lambda_{-1}$ with the nilpotent cone $\cN$.
For this, consider the fundamental correspondence specialized to the current situation
$$
\xymatrix{
\Loc_G(S,\partial_A S)& \ar[l]_-{\delta} \Loc_G(\wt S,  \partial_{\wt A} \wt S)\times_{\Gadj} \Badj \ar[r]^-{p}& \Loc_G(\wt S, \partial_{\wt A} \wt S)
}
$$

Given a geometric point $\rho \in \Loc_G(\wt S, \partial_{\wt A} \wt S)$  with monodromy $\rho(\gamma)\in \Gadj$
around the glued circles, 
one calculates 
$$
\xymatrix{
\SuppSp_{\Loc_G(\wt S, \partial_{\wt A} \wt S)}|_\rho\simeq \{ v\in \fg^* \; : \; Ad(\rho)v=v \}
}
$$
$$
\xymatrix{
\Lambda_{-1}|_\rho = \{ v\in \fg^* \; : \; \exists g\in \rho|_x, \; g\cdot \rho(\gamma) \in B,\; g\cdot v\in \mathfrak n \}
}
$$
i.e., there is a frame for the $G$-torsor given by the fiber of $\rho$ at $x\in S$ taking the monodromy around $\gamma$ into $B$
and the covector $v$ into $\mathfrak n$. 

Thus  $\cN$ evidently contains $\Lambda_{-1}|_\rho$; conversely, for any conjugacy class $[\alpha]\in \Gadj$ and 
$v\in\cN$ there exists a frame $g$ sending $\alpha$ to $B$ and $v$ to $\mathfrak n$.
\end{proof}


\section{Verlinde Loop Operators}\label{sect loop ops}

We record here the compatibility of the gluing of Corollary~\ref{main cor} with further natural symmetries
available in the Betti setting.
  
Let $Z(\cH_G) = \End_{\cH_G \ot \cH_G^\op}(\cH_G)$ be the center of the affine Hecke category.
Recall that $Z(\cH_G)$ is naturally an $E_2$-monoidal category with a universal central map $Z(\cH_G) \to \cH_G$.

We will recall the geometric description of  $Z(\cH_G)$ obtained in~\cite[Theorem 4.3.1]{spectral}.
 
 Let $\DCoh_{\proper/\Gadj}(\cL(\Gadj))$ denote the dg category of coherent sheaves
 on the loop space $\cL(\Gadj) \simeq \Loc_G(S^1 \times S^1)$
 with proper support over $\Gadj  \simeq \Loc_G(S^1)$. 
 
 Recall that convolution equips $\DCoh_{\proper/\Gadj}(\cL(\Gadj))$ with a natural
 $E_2$-monoidal structure.
 Recall the fundamental correspondence 
$$
\xymatrix{  \cL(\Gadj)   &\ar[l]_-p \cL(\Gadj)\times_{\Gadj} \Badj
\ar[r]^-\delta &
\Badj\times_{\Gadj} \Badj  
}
$$
and the induced functor
$$
\xymatrix{
\delta_*p^*: \DCoh_{\proper/\Gadj}(\cL(\Gadj))\ar[r] & \DCoh(\Badj \times_\Gadj \Badj) = \cH_G
}
$$

\begin{theorem}\cite[Theorem 4.3.1]{spectral} 
The functor $\delta_*p^*$ is the universal central map underlying an $E_2$-monoidal equivalence
$$
\xymatrix{
\DCoh_{prop/\Gadj}(\cL(\Gadj))\ar[r]^-\sim & Z(\cH_G) 
}
$$ 
\end{theorem}

\begin{remark}\label{rem central action}
It is useful to reformulate the universal central map of the theorem as a central action. 

Let $\cyl = S^1 \times [0,1]$ denote the cylinder, and $\gamma= S^1\times\{1/2\} \subset \cyl$ the meridian. 
Modifications of  local systems along $\gamma$ provides a correspondence
$$
\xymatrix{
 \Loc_G(S^1 \times S^1 ) \times \Loc_G(\cyl, \partial \cyl ) &\ar[l]_-{p_1} \Loc_G(\cyl \coprod_{\cyl\setminus \gamma} \cyl,
 \partial \cyl) \ar[r]^-{p_2}& \Loc_G(\cyl, \partial \cyl)
}$$
where the torus $S^1\times S^1 $ arises from  gluing a tubular neighborhood of $\gamma$ to itself along the complement of $\gamma$. 
The universal central map of the theorem extends to a central
$Z(\cH_G)$-action on $\cH_G$ with action map
 given by
 $$
\xymatrix{
 \cA \star \cM  = p_{2*}p_1^*(\cA \boxtimes \cM)
}
$$ 
(In order to construct the higher compatibilities for the action of $Z(\cH_G)$ as endomorphisms of the diagonal $\cH_G$-bimodule, we simply replace modifications along $\gamma= S^1\times\{1/2\} \subset \cyl$ by modifications along meridians $\gamma_i$ of cylinders in $S^1\times [0,1]$ labelled by configurations of little intervals in $[0,1]$.)

\end{remark}

\medskip

Now let us focus on the equivalence of Corollary~\ref{main cor}.

On the one hand, observe that $Z(\cH_G)$ naturally acts on the algebraic side 
$$
\xymatrix{
\Tr(\cH_G, \DCoh_\cN(\Loc_G(S,\partial_A S))) = 
\DCoh_\cN(\Loc_G(S,\partial_A S)) \otimes_{\cH_G\ot \cH_G^\op} \cH_G
}
$$
via its central action on the factor $\cH_G$ in the tensor product.

On the other hand,  as we will now explain, $Z(\cH_G)$ naturally acts on the geometric side 
$$ 
\xymatrix{
\DCoh_\cN(\Loc_G(\wt{S}, \partial_{\wt A} \wt S))
}
$$  by  what are called Verlinde loop operators. These are a direct generalization of the central action considered in
Remark~\ref{rem central action} immediately above.  
Recall the canonical curve and its tubular neighborhood $\gamma:S^1\times \{1/2\} \hookrightarrow S^1 \times [0,1] \subset \wt S$ coming from glued marked boundary components.  
Modifications of  local systems along $\gamma$ provides a correspondence
$$
\xymatrix{
 \Loc_G(S^1 \times S^1 ) \times \Loc_G(\wt S) &\ar[l]_-{p_1} \Loc_G(\wt S \coprod_{\wt S\setminus \gamma} \wt S) \ar[r]^-{p_2}& \Loc_G(\wt S)
}$$
where as in  Remark~\ref{rem central action}  the torus $S^1\times S^1 $ arises by gluing of a tubular neighborhood of $\gamma$ to itself along the complement of $\gamma$. More generally, we consider modifications of local systems along meridians of cylinders $S^1 \times I_i \subset S^1 \times [0,1]$ for arbitrary configurations of little intervals $\{I_i\}$ in $[0,1]$.
This provides a $Z(\cH_G)$-action on $\DCoh_\cN(\Loc_G(\wt{S}, \partial_{\wt A} \wt S))$ with action map 
$$
\xymatrix{
 \cA \star \cM  = p_{2*}p_1^*(\cA \boxtimes \cM)
}
$$

\begin{prop}
The equivalence of Corollary~\ref{main cor} respects the natural $Z(\cH_G)$-actions.
\end{prop}

\begin{proof}
This is a straightforward comparison of the correspondence of Remark~\ref{rem central action} with the correspondence defining Verlinde loop operators.  

Returning to the setting of Corollary~\ref{main cor}, it is convenient to express the glued surface in the form 
$\wt S =  S \coprod_{S^1 \coprod S^1} \cyl$ 
using the provided identifications $\partial_\alpha   S \coprod \partial_\beta S \simeq S^1 \coprod S^1 \simeq\partial\cyl$.

Now observe that the constructed equivalence 
$$
\xymatrix{
\DCoh_\cN(\Loc_G(S,\partial_A S)) \otimes_{\cH_G\ot \cH_G^\op} \cH_G \ar[r]^-\sim & \DCoh_\cN(\Loc_G(\wt{S}, \partial_{\wt A} \wt S))
}$$
is induced by the functor
$$
\xymatrix{
q_{2*}q_1^*: \DCoh_\cN(\Loc_G(S,\partial_A S)) \otimes  \DCoh(\Loc_G(\cyl, \partial\cyl)) \ar[r] & \DCoh_\cN(\Loc_G(\wt{S}, \partial_{\wt A} \wt S))
}$$
defined by the correspondence
\begin{equation}
\label{gluing diagram}
\xymatrix{
\Loc_G(S,\partial_A S) \times  \Loc_G(\cyl, \partial\cyl) &\ar[l]_-{q_1} \Loc_G(\wt S, \partial_{\wt A} \wt S) \times_{\Gadj\times \Gadj}\Badj\times\Badj
\ar[r]^-{q_2}& \Loc_G(\wt{S}, \partial_{\wt A} \wt S)
}\end{equation}
where the projection $\Loc_G(\wt S, \partial_{\wt A} \wt S)  \to \Gadj\times \Gadj$ is given by evaluation at the glued loops.

Now we can extend diagram~\eqref{gluing diagram} to also encode the modification of bundles along the distinguished curve
$\gamma = S^1 \times \{1/2\} \subset \cyl \subset \wt S$. Namely,
let us take the fiber product over  $\Loc_G(\cyl)$ of each term of diagram~\eqref{gluing diagram} with the following correspondence
\begin{equation}\label{verlinde op diagram}
\xymatrix{
 \Loc_G(S^1 \times S^1 ) \times \Loc_G(\cyl) &\ar[l]_-{p_1} \Loc_G(\cyl \coprod_{\cyl\setminus \gamma} \cyl) \ar[r]^-{p_2}& \Loc_G(\cyl)
}\end{equation}
Note that diagram~\eqref{verlinde op diagram} results from the correspondence of Remark~\ref{rem central action}  but without the $B$-reductions already found here in diagram~\eqref{gluing diagram}.

Finally, by base change, the natural $Z(\cH_G)$-actions given by $p_{2*}p_1^*$ are compatible 
with the gluing given by $q_{2*}q_1^*$. 

Repeating this argument for modifications at the meridians of $S^1 \times I_i \subset S^1 \times [0,1]$ for arbitrary configurations of little intervals $\{I_i\}$ in $[0,1]$ provides the higher compatibilities of this equivalence with monoidal structures. 

\end{proof}


\end{document}